\documentclass{amsart}[12pt, papera4]
\usepackage{tikz}
\usepackage{amssymb}
\usepackage{graphicx}
\usepackage{amsfonts}
\usepackage[all,2cell]{xy} \UseAllTwocells \SilentMatrices
\usepackage{graphicx}
\usepackage{amsthm}
\usepackage{amstext}
\usepackage{amsmath}
\usepackage{amscd}
\usepackage[mathscr]{eucal}
\usepackage{url}

\newtheorem*{theorem*}{Theorem}
\newtheorem{theorem}{Theorem}[section]
\newtheorem{proposition}[theorem]{Proposition}

\newtheorem{lemma}[theorem]{Lemma}

\theoremstyle{definition}
\newtheorem{definition}[theorem]{Definition}
\newtheorem{example}[theorem]{Example}

\theoremstyle{remark}

\numberwithin{equation}{section}

\def\N{{\mathbb N}}
\def\Z{{\mathbb Z}}

\def\R{{\mathbb R}}

\def\xy{{N_f(x,y,\delta)}}
\def\pf{{p \:\mathscr{F}}}
\def\f{{\mathscr{F}}}

\begin{document}
	\begin{Large}
		
		\title[On Stronger Forms of Expansivity]{On Stronger Forms of Expansivity}
		\author[S. H. Joshi]{Shital H. Joshi}
		\address{Department of Mathematics, Shree M. P. Shah Arts and Science College, Surendranagar, India\\
			Department of Mathematics, Faculty of Science, The Maharaja Sayajirao University of Baroda, Vadodara, India}
		\email{shjoshi11@gmail.com}
		\author[E. Shah]{Ekta Shah}
		\address{Department of Mathematics, Faculty of Science, The Maharaja Sayajirao University of Baroda, Vadodara, India}
		\email{shah.ekta-math@msubaroda.ac.in}

		

		\begin{abstract}
			\noindent We define the concept of stronger forms of positively expansive map and name it as $\pf-$expansive maps. Here $\f$ is a family of subsets of $\N$. Examples of positively thick expansive and positively syndetic expansive maps are constructed here. Also, we obtain conditions under which a positively expansive map is positively co--finite expansive and positively syndetic expansive maps. Further, we study several properties of $\pf-$expansive maps. A characterization of $\pf-$expansive maps in terms of $\f^*-$generator is obtained. Here $\f^*$ is dual of $\f$. Considering $(\Z,+)$ as a semigroup, we study $\f-$expansive homeomorphism, where $\f$ is a family of subsets of $\Z \setminus \{0\}$. We show that there does not exists an expansive homeomorphism on a compact metric space which is $\f_s-$expansive. Also, we study relation between $\f-$expansivity of $f$ and the shift map $\sigma_f$ on the inverse limit space.
		\end{abstract}
		
		\maketitle

		\section{Introduction}
		
		\noindent Let $X$ be a metric space and $f : X \longrightarrow X$ be a continuous map. Map $f$ is said to be \emph{positively expansive} if there exists $\delta >0$, known as an \emph{expansive constant}, such that for any two distinct points $x, \; y$ in $X$, there is a positive integer $n$ satisfying $d\left(f^n(x), f^n(y)\right)>\delta$ \cite{NA}. In 1952, Schwartzman introduced the notion of positively expansive maps \cite{SS}.   One of the basic problem here is existence/non-existence of expansive maps on a given metric space. Since its inception expansivity has been extensively studied in the areas of Topological Dynamics, Ergodic Theory, Continuum Theory, and Symbolic Dynamics. One of the good monograph for the study of expansive maps is \cite{HA}. 
		
		\medskip
		\noindent An important aspect of the study of expansivity is the study of its various generalizations and variations in different settings.  In 1970, Reddy, \cite{wr}, studied point--wise expansive maps whereas $h-$expansivity was studied by R. Bowen \cite{rb}. Kato defined and studied the notion of continuum--wise expansive homeomorphisms \cite{HK}. Shah studied notion of positive expansivity of maps on metric $G-$spaces \cite{es} whereas Barzanouni studied finite expansive homeomorphisms \cite{ali}.  Authors in \cite{ase} studied expansivity for group actions. In this paper we study stronger forms of expansivity.
		
		\medskip
		\noindent  Roughly, a map $f$ is positively expansive if given any two distinct point in the space $X$, there exists a positive integer $n$ such that the $n^{th}$ iterate of the two points under the map $f$ are separated significantly. The infinitude of the set of positive integers where this significant separation occurs is guaranteed due to the separation through positive iterates of $f$. In \cite{TKM}, Moothathu for the first time studied stronger forms of sensitivity on the basis of largeness of the subsets of the set of positive integers  where this significant separation occurs. Since every positively expansive map is sensitive if the space $X$ do not have isolated points, it makes sense to study expansivity through this infiniteness of the subsets of the set of positive integers.

		\medskip
		\noindent The paper is organized in following manner: In Section 2 we discuss preliminaries required for the content of the paper. Through Example \ref{eg1} we note that $\xy$ can be the empty set if the points gets separated by $0^{\text{th}}-$iterate of the map $f$. Because of this observation, we consider map to be positively expansive if the points gets separated by positive iterates and this further provide the condition under which $\xy$ is an infinite set. In Section 3 we define the stronger forms of positively expansive maps and name it as $\pf-$expansive maps. We give examples of positively thick expansive and positively syndetic expansive maps. Further, we obtain a condition under which a positively expansive map is positively co--finite expansive map and positively syndetic expansive map. Next, we prove many properties of $\pf-$expansive maps and give necessary examples to support the assumption. In Section 4, we characterize $\pf-$expansive maps through generators. Considering $(\Z, +)$ as a semigroup, in Section 5 we study $\f-$expansive homeomorphisms and show that there does not exists any expansive homeomorphism which is $\f_s-$expansive. We also, study relation between $\f-$expansivity of a homeomorphism $f$ and the shift map $\sigma_f$ on the inverse limit space.

		\section{Preliminaries}
		
		\medskip
		\noindent A subset $A$  of $\N$ is said to be a \emph{co--finite} if $\N \setminus A$ is finite,
		\emph{thick} if $A$ contains arbitrary large blocks of consecutive integers and \emph{syndetic} if  $A$ is infinite and there exists $M\in \N$ such that the gap between two consecutive integers in $A$ is bounded by $M$. A \emph{family} of subsets of $\N$ is any subset $\mathscr{F}$ of $\mathscr{P}(\N)$, the power set of $\N$, which is upward hereditary, i.e., if $B\in \mathscr{F}$ and $B\subset C\subset \N$ then $C\in \mathscr{F}$. The dual of a family $\mathscr{F}$ of subsets of $\N$ is denoted by $\mathscr{F}^*$ and is given by $\mathscr{F}^*=\{A\subset \N:A\cap B\neq \Phi \mbox{ for all }B\in \mathscr{F}\}$ \cite{Oprocha}. The families of co--finite, thick and syndetic subsets of $\N$ are denoted by $\mathscr{F}_{cf}$, $\mathscr{F}_t$, and $\mathscr{F}_s$ respectively. 
		
		\smallskip
		\noindent Let $f:X \longrightarrow X$ be a continuous map defined on a compact metric space $X$. A finite open cover $\mathscr{U}$ of $X$ is called a \emph{generator} for $f$ if for every bisequence $\{A_n\}$ of members of $\mathscr{U}$, $\bigcap_{n=0}^{\infty}f^{-n}(\bar{A_n})$ is at most one point and a \emph{weak generator} if $\bigcap_{n=0}^{\infty}f^{-n}(A_n)$ is at most one point. Here $\bar{A}$ denotes a closure of $A$. By an argument used in \cite[Theorem 2.2.3]{HA} it is easy to observe that $f$ is positively expansive if and only if $f$ has a generator. Fix $\delta>0$. A sequence $\{x_i:i\geq 0\}$ is said to be a \emph{$\delta-$pseudo orbit} if $d(f(x_i),x_{i+1})<\delta$ for all $i\geq 0$. For $x,\; y \in X$, a \emph{$\delta-$chain} from $x$ to $y$ is a finite $\delta-$pseudo orbit from $x$ to $y$. A map $f$ is said to be \emph{chain mixing} if for every $x, \; y\in X$ and $\delta>0$, there is $N\in \N$ such that for any $n\geq N$, there is $\delta-$chain of length $n$ from $x$ to $y$. A subset $Y$ of $X$ is called $invariant$ if $f(Y)\subseteq Y$. A map $f$ is called \emph{minimal} if each point $x\in X$ has a dense orbit in $X$. An invariant subset $Y\subset X$ is called minimal if the map $f_{|Y}:Y\longrightarrow Y$ is minimal map. A point $x\in X$ is called a \emph{minimal point} if it lies in some minimal subset of $X$. It is known that if $x$ is a minimal point of $f$ and $U$ is a neighbourhood of $x$ then the set $N_f(x,U)=\{n\in N:f^n(x)\in U\}$ is syndetic \cite{HF}. 
		
		\medskip
		\noindent Suppose $f:X\longrightarrow X$ is a homeomorphism defined on a compact metric space $X$. A closed subspace $X_f=\{\bar{x}=(x_i)_{i=-\infty}^{\infty}:x_i\in X \mbox{ and } f(x_i)=x_{i+1},\mbox{ for each }i\in \Z\}$ of $X^\Z$ is called the \emph{inverse limit space} generated by $f$. The homeomorphism $\sigma_f:X_f\to X_f$, defined by $\sigma_f \left((x_i)\right)=(y_i)$ such that $y_i=x_{i+1}=f(x_i)$ for all $i\in \Z$, is called the \emph{shift map} induced by $f$.

		\section{Positively $\mathscr{F}-$Expansive Maps}
		
		\medskip
		\noindent Let $X$ be a metric space and $f : X \longrightarrow X$ be a continuous map. For $x,\; y\in X$ and $\delta>0$ denote  
		$$N_f(x,y,\delta)=\left\{n\in \N:d\left(f^n(x),f^n(y)\right)>\delta\right\}.$$
		
		\noindent Then map $f$ is positively expansive with expansive constant $\delta$ if for any $x,\; y \in X$, $\xy \neq \Phi$. Note that $\xy \neq \Phi$ implies $x \neq y$.
		
		\smallskip
		\begin{example}\label{eg1}
			Consider the subspace $X=\left\{\frac{1}{n}, (1-\frac{1}{n}) : n \in \N\right\}$  of $\R$ with the usual metric. For  $x \in X$, let $x_+$  denote the element of $X$   which is immediately right to $x$ and $x_{-}$ that element of $X$  which is immediately left to  $x$. Define the map  $f:X \longrightarrow X$ by
			\[ f(x) = \left\{ \begin{array}{lll}
				x, & \mbox{if  $x\in\left\{0, \frac{1}{2}, 1\right\}$}\\
				x_+,  & \mbox{if  $x<\frac{1}{2}$}\\
				x_-,  & \mbox{if  $x> \frac{1}{2}$}\\
			\end{array}
			\right. \]
			Take $\delta$ such that $0< \delta< \frac{1}{6}$. The map $f$ is not positively expansive as $N\left(\frac{1}{2}, \frac{1}{3}, \delta\right) = \Phi$. Further, for any $x,\; y \in X\setminus\left\{0, \frac{1}{2}, \frac{1}{3}, \frac{2}{3}, 1\right\}$,  $\xy$ is always a finite set.
		\end{example}	
		
		\smallskip
		\noindent In view of Example \ref{eg1} it follows that the infinitude of $\xy$ is not guaranteed if the separation of distinct points occurs through $0^{th}-$iterate of $f$. However, if the significant separation occurs through the positive iterates of $f$, then the set $\xy$ is always an infinite set. For, suppose $f$ is a positively expansive map with expansive constant $\delta$. Then for $x,\; y \in X$ with $x\neq y$, $\xy \neq\Phi$.  Let $n \in \xy$. Then $f^n(x) \neq f^n(y)$. But this implies that $N_f\left(f^n(x),f^n(y),\delta\right)\neq\Phi$. If $k \in N_f\left(f^n(x),f^n(y),\delta\right)$, then $k+n \in \xy$. Continuing this way it can be observed that $\xy$ is always an infinite set. This motivates our definition of 
		positively $\mathscr{F}-$expansive map.
		
		\smallskip
		\begin{definition}
			Let $f:X\longrightarrow X$ be a continuous map defined on a metric space $X$ and $\f$ be a family of subsets of $\N$. Then $f$ is said to be a \emph{positively $\mathscr{F}-$expansive} if there exists $\delta>0$ such that for any two distinct points $x$ and $y$ in $X$ the set $N_f(x,y,\delta)\in \mathscr{F}$. The constant $\delta$ is called an \emph{$\mathscr{F}-$expansive constant} for $f$. Shortly we write positively $\mathscr{F}-$expansive as $\pf -$expansive.
		\end{definition}
		
		\noindent Note that for different families $\f$ of subsets of $\N$ we have different types of $\f -$expansive maps. For instance, if $\f=\f_t$, then $\pf -$expansivity is  positively thick expansivity or if $\f = \f_s$, then $\pf-$expansivity is positively syndetic expansivity. Further, every positively $\f -$expansive map is positively expansive. Also, if $f$ is $\pf_{cf}-$expansive then it is both $\pf_{t}-$expansive and $\pf_s-$expansive. In the following we give an example of a positively  $\f_t-$expansive map which is not $\pf_{cf}-$expansive.
		
		\medskip
		\begin{example}\label{ex1}
			An $\N-$block in $\N$ of length $m$ is a finite sequence of length $m$ with its elements from $\N$. We denote an $\N-$block of length $m$ by $t, t+1, t+2,\dots, t+m-1$. For instance, $1,2,3,4,5$ is a block of length $5$. Construct a subset $P$ of $\N$ as follows:
			Starting from $1$, put an $\N-$block of length $2$, leave a next $\N-$block of length $2$, add next an $\N-$block of length $4$, leave a next an $\N-$block of length 4. Continue like this by adding an $\N-$block of even length and then leaving an $\N-$block of same length. Therefore,  
			
			$$P=\{1, 2, 5, 6, 7, 8, 13, 14,.., 18, 25, 26,.., 32, 41, 42,.., 50, 61, 62,..,72,85,..\}$$
			
			\noindent whereas compliment of $P$, is given by 
			$$P^c=\{3, 4, 9, 10, 11, 12, 19,20,.., 24, 33, 34,.., 40, 51, 52,.., 60,..\}.$$
			
			\noindent Note that both $P$ and $P^c$ contains blocks of length $2n$, for all $n\in \N$. Therefore, both $P$ and $P^c$ are thick sets as both the sets contains blocks of every length.
			Set $b_1=2$. Define inductively $b_n$ by $b_n=b_{n-1}+n$. For instance, $b_2=4, \; b_3=7, \; b_4=11.$

			\noindent For $m \in \N$, let $B_m$ denote a first $\N-$block of length $2m$ in $P^c$. Suppose $B_m$ is given by $B_m=t, t+1, t+2,\dots, t+2m-1$. Further, let $C_m$ denote an $\N-$block of length $m+1$ in $B_m$ starting from $t$. Therefore $C_m=t, t+1, t+2,\dots,t+m$. For $r \in \N$, let $d_{2r}^{\; m}$ denote the number of $\N-$block of length $2r+1$ in $C_m$.  
			Let $2k$ be the highest integer such that $\N-$block of length $2k+1$ occurs for the first time in $C_m$ and satisfies the following three conditions:
			\begin{enumerate}
				\item $d_2^{\; m} + d_4^{\; m} + d_6^{\; m}+\dots+ d_{2k}^{\; m}=m-1$;
				\item $d_2^{\; m}>d_4^{\; m}>d_6^{\; m}>\dots>d_{2k}^{\; m}$;
				\item $d_{2l}^{\; m}\leq d_{2l}^{\; m+1}$, for all $l\geq 1$.
			\end{enumerate}
			
			\begin{center}
				\begin{figure}
					\includegraphics[scale=0.3]{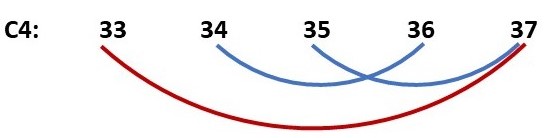} \includegraphics[scale=0.3]{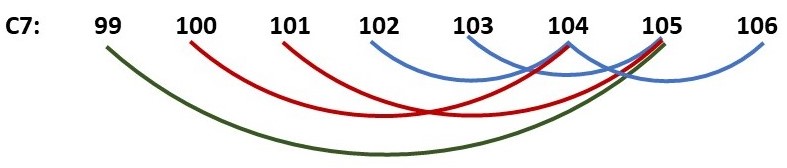}
					\caption{\lq $d_{2l}^m$ for $m=4$ and $m=7$ \rq}
				\end{figure}  
			\end{center}

			\noindent Using the Principle of Mathematical Induction we show that such a choice of $d_{2l}^{\; m}$ is possible. For $m=2$, $B_2=9,10,11,12$ and $C_2=9,10,11$. Therefore, the largest value of $2k$ is $2$ and $d_2^{\; 2}=1=2-1$ the only $\N-$block of length $2k+1=3$ is $C_2$ itself. Next, $B_3=19,20,21,22,23,24$, $C_3=19,20,21,22$ and the largest value of $2k$ is $2$. Therefore, there are exactly two $\N-$blocks of length $2k+1=3$, namely $19,20,21$ and $20,21,22$ and hence $d_2^{\; 3}=2=3-1$. Note that $d_2^{\; 2}\leq d_2^{\; 3}$ and $d_4^{\; 3}=0$. 
			
			\smallskip
			\noindent  Now, take $m=4=b_2$. Then $B_4=33,34,\dots,40$, $C_4=33,34,35,36,37$ and the largest value of $2k$ such that an $\N-$block of length $2k+1$ occurs in $C_m$ is $4$. Note that $C_4$ is an $\N-$block of length $5$ and therefore $d_4^{\; 4}=1$. Take $d_2^{\; 4}=2$. Then $d_2^{\; 4} + d_4^{\; 4}=2+1=3=4-1$. Also, $d_2^{\; 4}>d_4^{\; 4}$, $d_4^{\; 3}\leq d_4^{\; 4}$ and $d_4^{\; 3}\leq d_4^{\; 4}$.
			
			\smallskip
			\noindent Suppose that the given conditions are true for $m$. Note that for this $m$, there is $b_k$ and $r$ such that $m=b_k+r$, $0\leq r < b_{k+1}-b_k=k+1$. Here $2k$ is largest integer such that $\N-$block of length $2k+1$ occurs in $C_m$ satisfying the above three conditions. We show that conditions holds for $m+1$.  If $m+1 = b_k+r$, $0\leq r \leq k$, then take $d_{2s}^{\; m+1}=d_{2s}^{\; m}$ for $s\neq r$ and $d_{2r}^{\; m+1}=d_{2r}^{\; m}+1$. Observe that $d_{2}^{\; m+1}+d_{4}^{\; m+1}+...+d_{2k}^{\; m+1}=m-1+1=m$ and $d_{2s}^{\; m}\leq d_{2s}^{\; m+1}$, for $0\leq s \leq k$. Further, the condition $d_2^{\; m+1}>d_4^{\; m+1}>d_6^{\; m+1}>\dots>d_{2k}^{\; m+1}$ is satisfied by the choice of $d_{2s}^{\; m+1}$.

			\medskip
			\noindent Let $m+1=b_k+k+1=b_{k+1}$. Then the highest value of $2p$ such that an $\N-$block of length $2p+1$ occurs for the first time in $C_{m+1}$ is $2p=2(k+1)=2k+2$. Take $d_{2s}^{\; m+1}=d_{2s}^{\; m}$ for $0\leq s <k$ and $d_{2k+2}^{\; m+1}=1$. Note that all the three conditions are satisfied by the choice $d_{2s}^{\; m+1}$.

			\noindent Consider the shift space $(\Sigma_2, \sigma)$. Here $\Sigma_2=\left\{\bar{x}=\left(x_i\right) : x_i=0 \mbox{ or }x_i=1, \; i \in \N \right\}$ is considered with the metric $d\left(\bar{x}, \bar{y}\right)=\sum_{i=1}^{\infty}\frac{\left|x_i-y_i\right|}{2^i}$ and $\sigma : \Sigma_2\longrightarrow \Sigma_2$ is defined by $\sigma\left(\left(x_i\right)\right)=\left(y_i\right)$, where $y_i=x_{i+1}$, $i\geq 1$. We construct a point $\bar{x}=\left(x_i\right) \in \Sigma_2$ as follows: Put 
			\[ x_i = \left\{ \begin{array}{ll}
				1, & \mbox{if  $i\in P$ \& $i$ is odd }\\
				0,  & \mbox{if  $i \in P$ \& $i$ is even}\\
			\end{array}
			\right. \]

			\begin{center}
				\begin{figure}
					\hspace{-10pt} \includegraphics[scale=0.25]{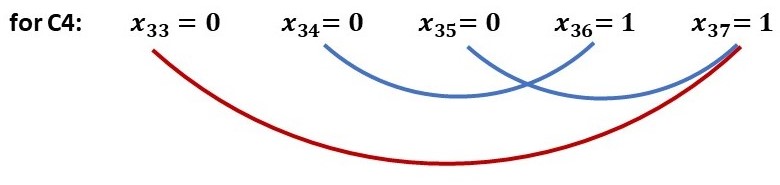}  \includegraphics[scale=0.25]{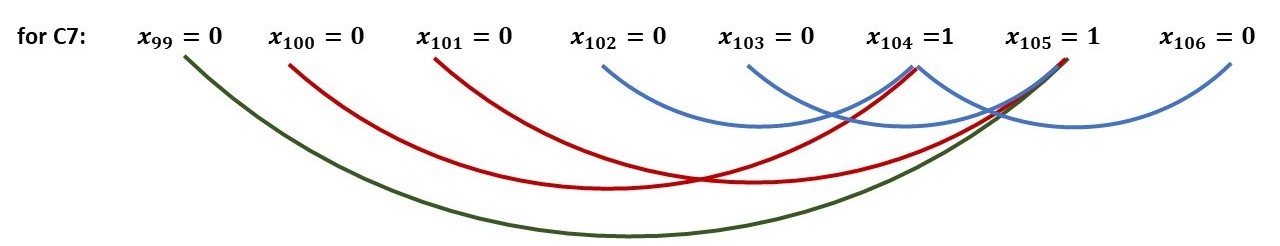}
					\caption{\lq Assignment of values of $x_i$ at different positions\rq}
				\end{figure}  
			\end{center}

			\noindent Now, let $m \in \N$ be such that $m\in P^c$ 
			and let $B_m=t, t+1, t+2,\dots, t+2m-1$. Put 
			\[ x_i = \left\{ \begin{array}{ll}
				0, & \mbox{if  $t+m+1 \leq i \leq t+2m-1$ \& $i$ is odd }\\
				1,  & \mbox{if  $t+m+1 \leq i \leq t+2m-1$ \& $i$ is even}\\
			\end{array}
			\right. \]
			
			\noindent Next, consider $C_m$ and let $2k$ be the largest integer such that an $\N-$block of length $2k+1$ occur in $C_m$ satisfying the three conditions. For $1\leq i \leq d_{2k}^{\; m}$, let $\N-$blocks of length $2k+1$ be given by $t+i-1, \; t+i, \;\dots, \;t+i+2k-1$. Put $x_{t+i-1} \neq x_{t+i+2k-1}$. Further, for $1 \leq j \leq d_{2k-2}^{\; m}$, an $\N-$block of length $2k-2+1$ is given by $t+d_{2k}^{\; m}+j-1, \; t+d_{2k}^{\; m}+j, \; \dots, \;t+d_{2k}^{\; m}+2k-2+j-1$. Put $x_{t+d_{2k}^{\; m}+j-1}\neq x_{t+d_{2k}^{\; m}+2k-2+j-1}$. Continuing this way $\N-$block of length $2\times 1 + 1$ will be given by $t+m-1-d_{2}^{\; m}+q-1, \; t+m-1-d_{2}^{\; m}+q, \; t+m-1-d_{2}^{\; m}+q+1$, $1\leq q\leq d_{2}^{\; m}$. Put $x_{t+m-1-d_{2}^{\; m}+q-1} \neq x_{t+m-1-d_{2}^{\; m}+q+1}$. Then 
			$$
			\begin{array}{lcl}
				\bar{x}&=&(10011010011010101000110110101010000111011010101010000110010 \\
				& &110101010101\dots).
			\end{array}
			$$ 
			
			\smallskip
			\noindent Take $X= O_{\sigma}(\bar{x}) \cup \left\{\overline{01}, \; \overline{10}\right\}$, where $\overline{01}=(01010101\dots)$. We show that $\sigma=\sigma_{|X}:X \longrightarrow X$ is $\pf_t-$expansive on $X$ with $\pf_t-$expansive constant $\delta$, where $0< \delta< \frac{1}{2}$. Let $\bar{t}, \; \bar{w}\; \in X$ with $\bar{t}\neq \bar{w}$.   
			
			\medskip
			\noindent \underline{Case--1:} Let $\overline{t} \neq \overline{w}$ be two elements in $X$. If $\overline{t}=\overline{01}$ and $\overline{w}=\overline{10}$, then $N\left(\overline{t}, \overline{w}, \delta\right)=\N$ and therefore it is a thick set. 
			
			\medskip
			\noindent \underline{Case--2:} Suppose $\overline{t}=\overline{01}$ and $\overline{w}=\sigma^r(x)$, for some $r\geq 0$. If $r$ is even then $N\left(\overline{t}, \overline{w}, \delta\right) \supset P\setminus \left\{1,2,\dots,r\right\}$ and therefore is a thick set. Let $m \in \N$ be such that $m\in P^c$ and let $B_m=t, t+1, t+2,\dots, t+2m-1$. Put 
			$$F=\bigcup_{m\geq 2}\left\{t+m+1, \; t+m+2,\;....\, t+2m-1\right\}.$$
			Then $F$ is a thick set. If $r$ is odd then $N\left(\overline{t}, \overline{w}, \delta\right) \supset F \setminus \{1,2,\dots,r\}$. 
			
			\medskip
			\noindent \underline{Case--3:} Take $\overline{t}=\overline{x}$ and $\overline{w}=\sigma^s(x)$, for some $s> 0$. Recall $P\subset \N$ is such that $P$ contains $\N-$block of length $2n$, for all $n\in \N$.
			
			\noindent If $s$ is odd, then $$N_{\sigma}(\overline{t},\overline{w},\delta) \supset P^{'}_{\; s} $$ where $P^{'}_{\; s}$ is a subset of $P$ obtained by removing last $s-$elements from each $\N-$block of length $2n$. For instance, if $s=3$, then $P^{'}_{\; 3}=\{5,13,14,15,25,26,\\ 27,28,29,41,\dots,47,61,\dots,70,\dots\}$. Note that $P^{'}_{\; s}$ is a thick set and therefore $N_{\sigma}(\overline{t},\overline{w},\delta)$ is thick. Next, suppose $s$ is even. Then $s=2p$ for some $p\in \mathbb{N}$. Let $C_m$ be the first $\N-$block of length $m+1$ in which a block of length $2p+1$ occurs for the first time. In this case $d_{2p}^m=1$. By definition of $d_{2p}^m$, it follows that an $\mathbb{N}-$block of length $1$ is in $N_{\sigma}(\overline{t},\overline{w},\delta)$. Consider $C_k$, $k\geq m$. Then as $k$ increases, the value of $d_{2p}^k$ will increase. Therefore if $d_{2p}^k=l$, then an $\mathbb{N}-$block of length $l$ is in $N_{\sigma}(\overline{t},\overline{w},\delta)$. Hence the set $N_{\sigma}(\overline{t},\overline{w},\delta)$ is thick. If $\overline{t}=\sigma^k(\overline{x})$ then by a similar argument as above $N_{\sigma}(\overline{t},\sigma^k(\overline{w}),\delta)$ is a thick set and $N_{\sigma}(\overline{t},\overline{w},\delta)\supset N_{\sigma}(\overline{t},\sigma^k(\overline{w}),\delta)$. Thus, in any case, the set $N_{\sigma}(\overline{t},\overline{w},\delta)$ is thick. Hence $\sigma$ is $\pf_t-$expansive. Note that in Case--3 for even $s$, compliment of $N_{\sigma}(\bar{t},\bar{w},\delta)$ contains an infinite set. Therefore $\sigma$ is not $\pf_{cf}-$expansive.
			
			\bigskip 
			\noindent In the following we give an example of $\pf_s-$expansive map.
			
			\medskip	
		\end{example}
		
		\begin{example}
			Consider the shift space  $\left(\Sigma_2, \sigma\right)$ as in Example \ref{ex1}. Let $\overline{x}$ be a periodic point of period $k$ in $\Sigma_2$.  Put $X=\overline{O_{\sigma}(x)}$ and $\sigma=\sigma_{|X}$. Then $\sigma$ is   $\pf_s-$expansive with $\f_s-$expansivity constant $\frac{1}{2}$. For if, $\overline{y}=(y_i)$ and $\overline{z}=(z_i)$ are two distinct points in $X$, then the set  $N_{\sigma}\left(y,z,\frac{1}{2}\right)=\{i\in \N:y_i\neq z_i\}$ contains a syndetic set $ k\N+m,\; 0 \leq m < k$ and therefore is a syndetic set.
		\end{example}
		
		\medskip
		\noindent Using following Lemma we obtain a condition under which a positively expansive map is $\pf_{cf}-$expansive map.
		
		\smallskip
		\begin{lemma} \cite{Chen} \label{L2}
			Suppose $(X,d)$ is a compact metric space and $f:X\to X$ is a continuous map. Given $\epsilon >0$ and $M>0$, there is $\delta >0$ such that every $\delta-$pseudo orbit $\{x_0,x_1,\dots,x_n,\dots\}$ satisfies $$d\left(f^M(x_k),x_{M+k}\right)\leq \epsilon$$ for all $k\geq 0$.
		\end{lemma}
		\begin{theorem}
			Let $X$ be a compact metric space. Suppose $f:X\longrightarrow X$ a positively expansive map. If $f$ is a chain mixing, then $f$ is $\pf_{cf}-$expansive map.
		\end{theorem}
		\begin{proof}
			Let $\delta>0$ be an expansivity constant for $f$ and let $x,\; y\in X$ with $x\neq y$. Then $\xy \neq \Phi$ and therefore there exists $k\in \xy$ such that $d\left(f^k(x),f^k(y)\right)>\delta$. Take $\epsilon$ such that $0<\epsilon <\delta.$
			Since $f$ is chain mixing, it follows that there is $n_0\in \N$ satisfying definition of chain mixing. Fix $M$ such that $M\geq n_0$. Then by Lemma \ref{L2}, for $\delta>0$ there exists $\beta >0$ such that every $\beta-$pseudo orbit $\{t_i\}$ satisfies 
			\begin{equation}\label{eq6}
				d\left(f^M(t_i),t_{M+i}\right)<\frac{\delta}{4}, \mbox{ for all }i\geq 0.
			\end{equation}
			Take $\theta$ such that $0<\theta \leq min\{\epsilon, \beta\}$. Then by chain mixing of $f$ there exists $\theta-$chain of length $M$ from $f^k(x)$ to $f^k(y)$ and vice-versa. Let these chains be given by $\{x_i\}=\{f^k(x)=x_0,x_1,\dots, x_M=f^k(y)\}$ and $\{y_i\}=\{f^k(y)=y_0,y_1,\dots, y_M=f^k(x)\}$. Using Equation \ref{eq6} for $i=0$ implies $d\left(f^{M+k}(x)\right),f^k(y))<\frac{\delta}{4}$ and $d\left(f^{M+k}(y)),f^k(x)\right)<\frac{\delta}{4} $. Therefore $$d\left(f^{M+k}(y),f^{M+k}(x)\right)>\delta-\frac{\delta}{4}-\frac{\delta}{4}=\frac{\delta}{2}$$ and hence $M+k\in N_f(x,y,\frac{\delta}{2})$. Infact, $N_f(x,y,\frac{\delta}{2})=\{M+k,M+1+k,\dots\}$ which is a co--finite set. Hence $f$ is $\pf_{cf}-$expansive map with $\f_{cf}-$expansivity constant $\frac{\delta}{2}$.		
		\end{proof}
		
		\smallskip 
		\noindent In the following theorem we obtain a condition under which a positively expansive map is a positively syndetic expansive map. Recall, if $x$ is a minimal point of $f$ and $U$ is any neighbourhood of $x$, then $N_f(x,U)\in \f_s$.
		
		\smallskip
		\begin{theorem}
			Let $X$ be a compact metric space. Suppose $f:X\longrightarrow X$ is positively expansive map such that the set of minimal points of $f$ is dense in $X$. Then $f$ is $\pf_s-$expansive map.
		\end{theorem}
		\begin{proof}
			Let $\delta>0$ be an expansivity constant for $f$ and let $x,\; y\in X$ with $x\neq y$. Positive expansivity of $f$ implies $\xy \neq \Phi$. Suppose $k\in \xy$. Then $d\left(f^k(x),f^k(y)\right)>\delta$. Since minimal points of $f$ are dense in $X$, minimal points of $f\times f$ are also dense in $X\times X$ \cite{EA}. We consider the following cases:
			
			\medskip
			\noindent \underline{Case--1:} Suppose $\left(f^k(x),f^k(y)\right)$ is a minimal point of $f\times f$. Let $U$ and $V$ be neighbourhoods of $f^k(x)$ and $f^k(y)$ respectively, such that $d(u,v)>\delta$ for every $u\in U$ and for every $v\in V$. Then, $U\times V$ is a neighbourhood of $\left(f^k(x),f^k(y)\right)$. Therefore the set $N_{f\times f}\left((f^k(x),f^k(y)),U\times V\right)\in \f_s$. Let $n\in N_{f\times f}\left((f^k(x),f^k(y)),U\times V\right)$. Then $d\left(f^{n+k}(x),f^{n+k}(y)\right)>\delta$ implies $N_{f\times f}\left((f^k(x),f^k(y)),U\times V\right)+k\subset \xy$. Hence $\xy\in \f_s$.
			
			\medskip
			\noindent \underline{Case--2:} Suppose $\left(f^k(x),f^k(y)\right)$ is not a minimal point of $f\times f$. By uniform continuity of $f^k$, there exists $\eta>0$ such that 
			\begin{equation*} 
				d(p,q)<\eta \implies d\left(f^k(p),f^k(q)\right)<\frac{\delta}{3}.	
			\end{equation*}
			Put $f^k(x)=a$ and $f^k(y)=b$. Then $a\neq b$ and there exists a sequence $\{(x_n,y_n)\}$ of minimal points of $f\times f$ converging to $(a,b)$. We show that for $m$ large enough $N_f(x_m,y_m,\delta)\subset N_f(x,y,\eta)$. If possible suppose $n\notin N_f(x,y,\eta)$. Then $d\left(f^n(x),f^n(y)\right)<\eta$ implies that $d\left(f^{k+n}(x),f^{k+n}(y)\right)<\frac{\delta}{3}$.
			
			\smallskip
			\noindent By uniform continuity of $f^n$ there exists $\theta$ with $0<\theta<\frac{\delta}{3}$ such that 
			$$d(p,q)<\theta \implies d\left(f^{n}(p),f^{n}(q)\right)<\frac{\delta}{3}.$$
			
			\smallskip 
			\noindent Since $x_n\to a$ and $y_n\to b$, as $n\to \infty$, it follows that there exists $m\in \N$ such that $d(x_m,a)<\theta$ and $d(y_m,b)<\theta$ and $x_m\neq y_m$. It is now easy to observe that $d\left(f^{n}(x_m),f^{n}(y_m)\right)<\delta$. Therefore $n\notin N_f(x_m,y_m,\delta)$. Hence, $N_f(x_m,y_m,\delta)\subset N_f(x,y,\eta)$.
			
			\smallskip
			\noindent Again $x_m\neq y_m$ and $f$ is positively expansive implies $N_f(x_m,y_m,\delta)\neq \Phi$. Suppose $l\in N_f(x_m,y_m,\delta)$. Since $(x_m,y_m)$ is minimal point, it follows that $\left(f^l(x_m),f^l(y_m)\right)$ is a minimal point and hence $N_f(x_m,y_m,\eta)\in \f_s$ by Case--1. Thus, $N_f(x,y,\eta)\in \f_s$. 
			
			\smallskip 
			\noindent Let $\beta=min\{\delta,\eta\}$. Then, $f$ is $\pf_s-$expansive with expansivity constant $\beta$.
		\end{proof}

		\medskip
		\noindent $\pf-$expansivity is preserved under the conjugacy if it is uniformly continuous as can be observed from the following result.
		
		\smallskip
		\begin{proposition}\label{Th5}
			Let $(X,d)$ and $(Y,\rho)$ be two metric spaces and $g :X \longrightarrow Y$ be a homeomorphism such that $g^{-1}$ is uniformly continuous on $Y$. Suppose $f$  is a $\pf-$expansive map on $X$. Then $gfg^{-1}$ is a $\pf-$expansive map on $Y$.
		\end{proposition}
		\begin{proof}
			Let $\delta$ be an $\f-$expansivity constant for $f$. Then uniform continuity of $g^{-1}$ on $Y$ implies there exists $\epsilon>0$ such that 
			\begin{equation}\label{eq8} 
				d(w,t)\geq \delta \implies \rho \left(g(t), g(w)\right)\geq \epsilon.
			\end{equation}
			Let $\eta$ be such that $0< \eta < \epsilon$. We show that $h=gfg^{-1}$ is $\pf-$expansive with $\f-$expansivity constant $\eta$. If $x$ and $y$ are distinct points of $Y$ then $g^{-1}(x)$ and $g^{-1}(y)$ are distinct points of $X$. Therefore $\pf-$expansivity of $f$ implies $N_f\left(g^{-1}(x), \; g^{-1}(y), \;\delta\right)\in \f$. Take $n \in N_f\left(g^{-1}(x), \; g^{-1}(y), \;\delta\right)$. Then by Equation \ref{eq8} it follows that $n \in N_{h}(x,\; y,\;\eta)$. Therefore  $N_f\left(g^{-1}(x),\;g^{-1}(y), \;\delta \right) \\ \subset N_h(x,\; y,\; \eta)$ as for any $k\in \N$, $(gfg^{-1})^k=gf^kg^{-1}$. Hence $N_h(x, \; y, \; \eta) \in \f$.
		\end{proof}
		
		\medskip
		\noindent In the Theorem \ref{Th5}, uniform continuity of $g^{-1}$ is a necessary condition can be observed from following Example. Recall, there does not exists any positively expansive map on $[0, 1)$, where $[0,1)$ is considered with the usual metric. Therefore, there does not exists $\pf-$expansive map on $[0,1)$.
		
		\smallskip
		\begin{example}
			Consider $X=[0, \infty)$ with usual metric. Define $f:X \longrightarrow X$ by $f(x)=2x$. Then $f$ is a $\pf-$expansive map with $\pf-$expansive constant $\delta$, $0 < \delta < 1$. Define $g: [0,\infty) \longrightarrow [0,1)$ by $g(x)=\frac{x}{x+1}$. Then $g$ is a homeomorphism and $h=gfg^{-1} : [0,1) \longrightarrow [0,1)$ given by $h(y)=\frac{2y}{y+1}$ is not $\pf-$expansive as $g$ is not positively expansive. Note that $g^{-1}$ is not uniformly continuous. 
		\end{example}
		
		\medskip 
		\noindent In the following Theorem we prove that product of two (and hence finite) $\pf-$expansive maps is again  $\pf-$expansive and also converse holds.
		\begin{proposition}
			Product of two $\pf-$expansive maps is $\pf-$expansive if and only if each map is $\pf-$expansive. 
		\end{proposition}
		\begin{proof}
			Let $(X,d)$ and $(Y, \rho)$ be two metric spaces and $f:X \longrightarrow X$ and $g:Y \longrightarrow Y$ be $\pf-$expansive maps with respective  $\f-$expansivity constants $\delta$ and $\eta$. For $x=(x_1,x_2)$ and $y=(y_1,y_2)$ in $X\times Y$, $D(x,y)$ is given by 
			$$D(x,y)=max\left\{d(x_1,y_1),\rho(x_2,y_2)\right\}.$$ 
			Let $h=f \times g$. We will show that $h$ is $\pf-$expansive on $X\times Y$ with $\f-$expansive constant $\beta$, where $0< \beta < min\{\delta, \eta\}$. Let $x=(x_1,x_2)$ and 
			$y=(y_1,y_2)$ be two distinct points in $X\times Y$. Then $x_i\neq y_i$, for some $i \in \{1, 2\}$. Suppose $x_1\neq y_1$. Then $\pf-$expansivity of $f$ implies $N_{f}(x_1,y_1,\delta) \in \f$. It is easy to observe that $N_{f}\left(x_1,y_1,\delta\right) \subset N_h\left(x,y,\beta\right)$ and therefore $N_h\left(x,y,\beta\right) \in \f$. Hence $h$ is $\pf-$expansive. 
			
			\smallskip
			\noindent Conversely, suppose $h=f \times g$ is $\pf-$expansive with $\f-$expansivity constant $\delta$. Let $x_1, \; y_1 \in X$ such that $x_1\neq y_1$. For $x_2\in Y$, let $x=(x_1,x_2), \; y=(y_1,x_2) \in X\times Y$. Then $x\neq y$. Since $h$ is $\pf-$expansive it follows that $N_h(x,y,\delta)\in \f$. It is easy to observe that $N_{h}\left(x,y,\delta\right) \subset N_f\left(x_1,y_1,\delta\right)$ and therefore $N_f\left(x_1,y_1,\delta\right) \in \f$. Hence $f$ is $\pf-$expansive. Similarly, one can show that $g$ is $\pf-$exapansive. \end{proof}

		\medskip 
		\noindent In \cite{Bry}, Bryant showed that if $A$ is subset of $X$ such that $X\setminus A$ is finite and $f$ is a homeomorphism on $X$ such that $f$ is expansive on $A$, then it is expansive on $X$. In the following result, we obtain similar results for $\pf-$expansive maps.
		
		\smallskip
		\begin{theorem}
			Let $X$ be a metric space and $f:X \longrightarrow X$ be a continuous map. Suppose $A$ is a subset of $X$ such that $X\setminus A$ is finite. If $f$  is $\pf-$expansive on $A$, then  $f$ is $\pf-$expansive on $X$.
		\end{theorem}
		\begin{proof}
			Let $\delta$ be an $\f-$expansive constant for $f$ on $A$ and let $x \in X\setminus A$.  It is sufficient to show that $f$ is $\pf-$expansive on $A\cup \{x\}$. Let $a,\; b$ be any two distinct points in $A\cup \{x\}$. If $a,\; b\in A$ then we are through as $f$ is $\pf-$expansive on $A$. Suppose $b=x$. Then we show that there is at most one $y\in A$ such that $$d\left(f^m(x),f^m(y)\right)\leq \frac{\delta}{2}, \mbox{ for all } m\in \N.$$
			If possible suppose there are two such points $y$ and $z$ in $A$, then by triangle inequality
			$$ d\left(f^m(y),f^m(z)\right) \leq \delta	\mbox{ for all } m\in \N,$$
			which is a contradiction to $\pf-$expansivity of $f$ on $A$.
			
			\medskip
			\noindent Let $y$ be a point as described above and let $\eta$ be such that $0<\eta<d(x,y)\leq \frac{\delta}{2}$. Take $p\in A$ such that 
			$p\neq y$. Then $N_f(y,p,\delta)\in \f$. Further, for $k\in N_f(y,p,\delta)$,
			it is easy to observe that $d\left(f^k(p),f^k(x)\right)\geq \frac{\delta}{2} > \eta$ and therefore $k\in N_f(p,x,\eta) $. But this implies, $N_f(p,x,\eta)\in \f$. Next, if there is no such $y$, then $d\left(f^m(x),f^m(z)\right)\geq \frac{\delta}{2} > \eta$, for all $z\in A$ and $m\in \N$. In particular, for $z=p$, $N_f(p,x,\eta)=\N$. Thus, in any case $N_f(p,x,\eta)\in \mathscr{F}$. So, $\eta$ is an $\f-$expansivity constant for $f$ on $A\cup \{x\} $. 
		\end{proof}
		
		\smallskip
		\noindent In the following we relate positively thick expansivity of $f$ and its powers $f^n$. Note that by a similar argument one can relate positively syndetic expansivity of $f$ and $f^n$.
		
		\smallskip
		\begin{theorem} \label{T1}
			Let $X$ be a metric space and let $f:X \longrightarrow X$ be a uniformly continuous map. Then $f$ is $\pf_t-$expansive if and only if $f^m$ is $\pf_t-$expansive for $m\in \N$.
		\end{theorem}
		\begin{proof}
			Let $f$ be $\pf_t-$expansive with $\f-$expansivity constant $\delta$. Fix $m\in \N$. By uniform continuity of $f$ there exists $\eta>0$ such that for $t,\; w\in X$	\begin{equation} \label{eq1}
				d\left(f^i(t),f^i(w)\right)\geq \delta \implies d(t,w)>\eta.
			\end{equation} 
			We show that $f^m$ is $\pf_t$-expansive with $\f_t-$expansivity constant $\eta$. Let $x,\; y\in X$ with $x\neq y$. Then $f$ is $\pf_t-$expansive implies $A=\xy \in \f_t$. For $n\in A$ there exists $r\in \N$ such that $n=rm+p$, for some $p$, $0<p\leq m$. Put $B=\{r\in \N:0<n-rm\leq m,\; n\in A\}$.
			
			\noindent CLAIM: $B\in \f_t$.
			
			\noindent Let $k\in \N.$ Then $A\in \f_t$ implies there is a block of length $km$ in $A$. Suppose this block is $$\{a,a+1,\dots, a+m,\dots, a+2m,\dots, a+km\}\subset A.$$
			By Division Algorithm, for this $a\in A$, there is $r\in \N$ such that $\frac{\displaystyle{a}}{\displaystyle{m}}-1 \leq r < \frac{\displaystyle{a}}{\displaystyle{m}}$. Therefore $r\in B$. Further, $\frac{\displaystyle{a}}{\displaystyle{m}} \leq r+1 < \frac{\displaystyle{a}}{\displaystyle{m}}+1$. Hence $a+m\in A$, implies $r+1\in B$. In general, $r+i\in B$ for $a+im\in A$, $0\leq i\leq k$. Hence $\{r,r+1,\dots r+k\}\subset B$ and thus $B$ has a block of length $k$. But $k\in \N$ is arbitrarry. Therefore $B\in \f_t$.
			Again, let $n\in A$ with $n=rm+p$. Then Equation $\ref{eq1}$ implies $$d\left(f^{rm}(x),f^{rm}(y) \right)>\eta.$$
			Therefore, $r\in N_{f^m}(x,y,\eta)$. Further, it is easy to observe that $B\subset N_{f^m}(x,y,\eta)$. But $B\in \f_t$. Therefore $N_{f^m}(x,y,\eta)\in \f_t$. Hence $f^m$ is $\f_t-$expansive.
			
			\smallskip
			\noindent Conversely, suppose $f^m$ is $\pf_t$-expansive with $\f_t-$expansivity constant $\delta$. Uniform continuity of $f$ implies that there exists $\eta >0$ such that for $t,w\in X$ \begin{equation} \label{eq2}
				d\left(f^i(t),f^i(w)\right)\geq \delta \implies d(t,w)>\eta, \mbox{ } 0\leq i\leq m-1.
			\end{equation}
			Let $x, \; y\in X$ with $x\neq y$. Then $A=N_{f^m}(x,y,\delta)\in \f_t$. Put $$B=\bigcup_{q=0}^{m-1}(mA-q).$$
			Here $mA-q=\{ms-q:s\in A\}$.
			
			\noindent CLAIM: $B\in \f_t$.
			
			\noindent Let $k\in \N$. Then there exists $j\in \N$ such that $jm^2>k$. It is sufficient to show, there is a block of length $jm^2$ in $B$. Now $A$ is a thick set. Therefore there is a block of length $jm$ in $A$, say, $a,a+1,\dots,a+m,\dots,a+jm-1.$ Note that for each $a+u\in \{a,a+1,\dots,a+jm-1\}\subset A$, the block of length $m$ given by $$ m(a+um-1), m(a+um-1)-1,\dots, m(a+um-1)-(m-1)$$ is in $B$. Therefore $$ma-(m-1), ma-(m-2),\dots,ma-1,ma,ma+1,\dots,ma+m,\dots,m(a+jm-1)$$ is a block of length $nm^2$ in $B$. Hence $B\in \f_t$.
			
			\smallskip
			\noindent In order to show that $N_f(x,y,\eta)\in \f_t$ it is sufficient to show that $B\subset N_f(x,y,\eta)$. Let $b\in B$ with $b=ma-q$, for some $a\in A$, for some $q$, $0\leq q\leq m-1$. Now $a\in A$ implies $d(f^{ma}(x),f^{ma}(y))>\delta$. Therefore, Equation \ref{eq2} implies $$d\left(f^b(x),f^b(y)\right)=d\left(f^{ma-q}(x),f^{ma-q}(y)\right)>\eta,$$ which further implies $b\in N_f(x,y,\eta)$. Hence the proof. 	
		\end{proof}
		
		\section{Generator and $\pf-$Expansivity}
		\noindent Keynes and Robertson introduced the notion of generator in \cite{Key} and characterized expansive homeomorphisms using generators. In this Section we obtain similar results for $\pf-$expansive maps. 
		
		\begin{definition}
			Let $X$ be a metric space and $f:X \longrightarrow X$ be a continuous map. A finite open cover $\mathscr{U}$ of $X$ is said to be an \emph{$\f-$generator} for $f$ if for every sequence $\{A_n\}$ of members of $\mathscr{U}$, $\bigcap_{n\in S}f^{-n}(\bar{A_n})$ contains at most one point, for all $S\in \f$. Here $\bar{A}$ denotes closure of $A$ in $X$.
		\end{definition}
		\begin{definition}
			Let $X$ be a metric space and $f:X \longrightarrow X$ be a continuous map. A finite open cover $\mathscr{U}$ of $X$ is said to be a \emph{weak $\f-$generator} for $f$ if for every sequence $\{A_n\}$ of members of $\mathscr{U}$, $\bigcap_{n\in S}f^{-n}(\bar{A_n})$ is at most one point for all $S\in \f$.
		\end{definition}
		
		\noindent Obviously if $\mathscr{U}$ is an $\f-$generator then $\mathscr{U}$ is a weak $\f-$generator. Suppose $f:X\to X$ is a $\pf-$expansive map with $\f-$expansivity constant $\delta$. Then for $x, \; y\in X$ with $x\neq y$, $\xy \in \f$. Equivalently, $f$ is $\pf-$expansive implies if $\xy \notin \f$, then $x=y$. Recall, for a family $\f$, dual of $f$ is denoted by $\f^*$ and is given by $\f^*=\{A\subset \N:A\cap B\neq \Phi \mbox{ for all }B\in \f\}$. In the following Theorem we show that $\pf-$expansive map has $\f^*-$generator.
		
		\begin{theorem} \label{T6}
			Let $X$ be a compact metric space and $f:X\to X$ be a continuous map. Suppose $f$ is $\pf-$expansive. Then $f$ has an $\f^*-$generator. Here $\f^*$ is the dual of $\f$.
		\end{theorem}
		\begin{proof}
			Suppose $f$ is $\pf-$expansive with $\f-$expansivity constant $\delta$. We show that $f$ has $\f^*-$generator. $A\in \f^*$ provided that $\N\setminus A \notin \f$. Let $\mathscr{U}$ be a finite open cover of $X$ consisting of open balls of radius $\frac{\delta}{2}$ and let $S\in \f^*$. If possible suppose $x,y\in \bigcap_{n\in S}f^{-n}(\bar{A_n})$, where $\{A_n\}_{n\in S}$ is a sequence from elements of $\mathscr{U}$. Then for $n\in S$, $$d(f^n(x),f^n(y))\leq \delta,$$
			which further implies $\xy \cap S=\Phi$. But $\f^*$ is dual of $\f$. Therefore $\xy \notin \f$. Since $f$ is $\pf-$expansive with $\f-$expansivity constant $\delta$, it follows that $x=y$. Hence $f$ has $\f^*-$generator.
		\end{proof}
		
		\medskip 
		\noindent In the following Theorem we show that converse of above result is true.
		
		\begin{theorem} \label{T7}
			Let $X$ be a metric space and $f:X \longrightarrow X$ be a continuous map. Suppose $f$ has weak $\f-$generator. Then $f$ is $\pf^*$-expansive.
		\end{theorem}
		\begin{proof}
			Suppose $\mathscr{U}$ is a weak $\f-$generator for $f$. Let $\delta$ be a Lebesgue number for $\mathscr{U}$. If $\f^*$ denote dual of $\f$, then we show that $f$ is $\pf^*-$expansive with $\f^*-$expansivity constant $\delta$. Suppose $x, \; y\in X$ and if possible suppose $\xy \notin \f^*$. Then there exists $A\in \mathscr{F}$ such that $N_f(x,y,\delta)\cap A=\Phi$. Therefore $d\left(f^n(x),f^n(y)\right)\leq \delta$, for all $n\in A$. This further implies that for each $n\in A$, there exists $A_n\in \mathscr{U}$ such that $\{f^n(x),f^n(y)\}\subset A_n$. Hence $x,y \in \bigcap_{n\in A}f^{-n}(A_n)$. But $\mathscr{U}$ is a weak $\f-$generator. Therefore $x=y$.
		\end{proof}
		
		\noindent As a consequence of the Theorem \ref{T6} and Theorem \ref{T7} following holds:
		\begin{theorem} \label{T4}
			Let $X$ be a compact metric space and $f:\longrightarrow X$ be a continuous map. Then the following are equivalent:
			\begin{itemize}
				\item [1.] $f$ is $\pf-$expansive.
				\item[2.] $f$ has $\f^*-$generator.
				\item[3.] $f$ has weak $\f^*-$generator.
			\end{itemize}
		\end{theorem}
		\section{$\mathscr{F}$-expansivity}
		
		\noindent Let $(S, \ast)$ be a topological semigroup. For $x, \; y \in S$, we denote $x\ast y$ by $xy$. A subset $T$ of $S$ is said to be \emph{syndetic} if there is a compact subset $K$ of $S$ such that for every $s\in S$, there is $k\in K$ with $ks\in T$. A subset $W$ of $S$ is said to be \emph{thick} if for any compact subset $K$ of $S$, there is $s\in S$ such that $Ks\subset W$ \cite{RT}. Here $Ks = \{ks : k\in K\}$. A family of subsets of $S$ is any subset of $P(S)$, power set of $S$, which is upward hereditary. Family of thick subsets of $S$ is denoted by $\f_t$ and family of syndetic subsets of $S$ is denoted by $\f_s$.
		
		\smallskip
		\noindent Consider semigroup $(\Z,+)$. Then a thick subset of $\N$ is also a thick subset of $\Z$. Further note that a syndetic subset of $\Z$ is neither bounded below nor bounded above. For, if $A\subset \Z$ is syndetic set which is either bounded below or bounded above, then $\Z \setminus A$ will contain blocks of length $n$ for all $n\in \N$ and hence $\Z\setminus A$ is a thick set, which is not possible.
		
		\smallskip
		\noindent Recall, a homeomorphism $f:X \longrightarrow X$ is said to be \emph{expansive} if there exists $\delta>0$ such that for $x \neq y$ in $X$, $N_f(x,y,\delta)\neq \Phi$. The constant $\delta$ is called an \emph{expansivity constant} for $f$. Here $\xy=\{n\in \Z\setminus \{0\}: d\left(f^n(x),f^n(y)\right)>\delta\}$. Note that $\xy$ is an infinite set if $f$ is an expansive homeomorphism.
		
		\begin{definition}
			A homeomorphism $f:X \longrightarrow X$ is said to be \emph{$\mathscr{F}$-expansive} if there exists $\delta>0$ such that for any two distinct points $x$ and $y$ in $X$ the set $N_f(x,y,\delta)\in \mathscr{F}$,  where $\mathscr{F}$ is a family of subsets of $\Z\setminus \{0\}$. The constant $\delta$ is called an $\f-$expansivity constant for $f$.
		\end{definition}
		
		\begin{definition} \cite{Bry} Let $f$ be an expansive homeomorphism on a compact metric space $X$. Two points $x, \; y \in X$ are said to be \emph{positively (negatively) asymptotic} if for each $\epsilon>0$, there exists an integer $N$ such that $n>N (n<N)$ implies $d(f^n(x),f^n(y))< \epsilon$.	
		\end{definition}
		
		\begin{lemma}\label{L1}\cite{BF1}
			Let $X$ be a compact metric space and $f:X \longrightarrow X$ be an expansive homeomorphism on $X$, then there exists $x,\; y,\; z,\; w\in X$ such that $x$ and $y$ are positively asymptotic under $f$ and $z$ and $w$ are negatively asymptotic under $f$.
		\end{lemma}
		
		\begin{theorem}
			There is no expansive homeomorphism on a compact metric space which is $\mathscr{F}_s-$expansive homeomorphism.
		\end{theorem}
		\begin{proof}
			Let $X$ be a compact metric space and $f:X \longrightarrow X$ be an expansive homeomorphism with expansive constant $\delta$. Then by  Lemma \ref{L1} there exist $x,\; y,\; z,\; w\in X$ such that $x$ and $y$ are positively asymptotic and $z$ and $w$ are negatively asymptotic for $f$. Let $\epsilon >0$ be given. Then $x$ and $y$ are positively asymptotic implies there exist $N\in \Z$ such that for all $n>N$ implies $d\left(f^n(x),f^n(y)\right)\leq \epsilon$. Therefore 
			\begin{equation}\label{eq7}
				\{N+1,N+2,N+3,\dots\}\cap N_f(x,y,\epsilon)=\Phi. 
			\end{equation}	
			Note that $\{N+1,N+2,N+3,\dots\}$ is a thick set in $\Z$ as it is a thick set in $\N$. Hence Equation \ref{eq7} implies $N_f(x,y,\epsilon)\notin \f_s$ for any $\epsilon >0$. Therefore $f$ is not $\f_s-$expansive.
		\end{proof}
		\noindent In the following result we will see the relationship between $\mathscr{F}$-expansivity of a homeomorphism and the shift map on inverse limit space. 
		\begin{theorem}
			Let $X$ be a compact metric space and $f:X \longrightarrow X$ be a homeomorphism. Then $f$ is $\mathscr{F}-$expansive if and only if the shift map $\sigma_f:X_f \longrightarrow X_f$ is $\mathscr{F}$-expansive. Here $\mathscr{F}$ is a family of subsets of $\mathbb{Z}$.
		\end{theorem}
		\begin{proof}
			Suppose $f$ is $\mathscr{F}-$expansive with $\f-$expansivity constant $\delta >0$. We show that $\sigma_f$ is $\f-$expansive with $\f-$expansivity constant $\delta$. Let $\bar{x}=(x_i)_{i\in \Z}$ and $\bar{y}=(y_i)_{i\in \Z}$ be two distinct points in $X_f$. Then $x_j\neq y_j$ for some $j\in \Z$, and $\f-$expansivity of $f$ implies that $N_f(x_j,y_j,\delta)\in \mathscr{F}$. Let $k\in N_f(x_j,y_j,\delta)$. Then 
			$$ \bar{d}\left(\sigma_f^k(\sigma_f^j(\bar{x})),\sigma_f^k(\sigma_f^j(\bar{y}))\right)=\sum_{i=-\infty}^{\infty}\frac{\displaystyle{d\left(f^k(\sigma_f^j(x)_i),f^k(\sigma_f^j(y)_i)\right)}}{2^{|i|}}>\delta. $$
			Therefore, $k\in N_{\sigma_f}(\sigma_f^j(\bar{x}),\sigma_f^j(\bar{y}),\delta)$. But this implies
			\begin{equation}\label{eq5}
				N_{\sigma_f}(\sigma_f^j(\bar{x}),\sigma_f^j(\bar{y}),\delta)+j\subset N_{\sigma_f}(\bar{x},\bar{y},\delta).
			\end{equation}  
			Hence $N_f(x_j,y_j,\delta)\subset N_{\sigma_f}(\sigma_f^j(\bar{x}),\sigma_f^j(\bar{y}),\delta)$. Therefore $N_{\sigma_f}(\sigma_f^j(\bar{x}),\sigma_f^j(\bar{y}),\delta) \in \f$ and Equation \ref{eq5} implies $N_{\sigma_f}(\bar{x},\bar{y},\delta)\in \f$. But $\bar{x},\; \bar{y} \in X_f$ are arbitrary. Thus, $\sigma_f$ is $\f-$expansive.
			
			\smallskip
			\noindent Conversely, suppose $\sigma_f$ is $\f-$expansive with $\f-$expansivity constant $\delta$. Let $\alpha$ be the diameter of $X$. Choose $k\in \N$ such that $0<\frac{\alpha}{2^k}<\frac{\delta}{3}$. By uniform continuity of $f$ there exists $\eta>0$ such that 
			\begin{equation} \label{eq4}
				d(t,w)<\eta \implies d\left(f^i(t),f^i(w)\right)<\frac{\delta}{9}, \mbox{ for all }i \mbox{, }|i|\leq k.
			\end{equation}
			Let $x,\;y\in X$ such that $x\neq y$. Take $\bar{x}=(\dots,x_{-1},x_0=x,x_1,\dots)$, $\bar{y}=(\dots,y_{-1},y_0=y,y_1,\dots) \in X_f$. Then $\bar{x}\neq \bar{y}$. Therefore $\f-$expansivity of $\sigma_f$ implies $N_{\sigma_f}(\bar{x},\bar{y},\delta)\in \f$. We complete the proof by showing that  $N_{\sigma_f}(\bar{x},\bar{y},\delta) \subset N_f(x,y,\eta)$. Suppose $n\notin N_f(x,y,\eta)$. Then $d(f^n(x),f^n(y))<\eta$. Equation \ref{eq4} implies $$d\left(f^i(f^n(x)),f^i(f^n(y))\right)<\frac{\delta}{9},$$ for all $i$ with $|i|\leq k$. It is easy to verify that $\bar{d}\left(\sigma_f^n(\bar{x}),\sigma_f^n(\bar{y})\right)<\delta$.	Thus, $n\notin N_{\sigma_f} (\bar{x},\bar{y},\delta)$. Hence, $N_f(x,y,\eta) \subset N_{\sigma_f}(\bar{x},\bar{y},\delta)$ and therefore, $N_f(x,y,\eta)\in \f$. Thus, $f$ is $\f-$expansive with $\f-$expansivity constant $\eta$.
		\end{proof}

	\end{Large}

\medskip
		
		\begin{thebibliography}{00}
			
			
			\bibitem{EA} E. Akin, E. Glasner, \textit{Residual properties and almost equicontinuity}, J. Anal. Math. 84 (2001), 243--86, https://doi.org/10.1007/BF02788112.
			
			\bibitem{NA} N. Aoki, \textit{Topological Dynamics} in: K. Morita and J. Nagata, (eds.), Topics in General Topology, North-Holland, Amsterdam, 1989, 625-740.
			
			\bibitem{HA} N. Aoki and K. Hiraide, \textit{Topological theory of dynamical systems}, Recent Advances, N. H. publications, 1994.
			
			\bibitem{ali} A. Barzanouni, \textit{Finite expansive homeomorphisms}, Top. \& its Appl. 253 (2019) 95--112, https://doi.org/10.1016/j.topol.2018.11.018.
			
			\bibitem{ase} A. Barzanouni, M. S. Divandar, E. Shah, \textit {On Properties of Expansive Group Actions}, Acta Mathematica Vietnamica 44 (2019) 923--934, https://doi.org/10.1007/s40306-019-00330-9.
			
			\bibitem{rb} R. Bowen, \textit{Entropy--expansive maps}, Trans. Amer. Math. Soc. 164 (1972) 323--331.
			
			\bibitem{Bry} B. F. Bryant, \textit{Expansive self-homeomorphisms of a compact metric space}, American Mathematical Monthly 69 (5) (1962) 386--391, https://doi.org/10.1080/00029890.1962.11989902.
			
			\bibitem{BF1} B. F. Bryant, \textit{On expansive homeomorphisms}, Pacific J. Math. 10 (1960) 1163--1167.
			
			\bibitem{Chen} L. Chen, S. Li, \textit{Shadowing property for inverse limit spaces}, Proc. Amer. Math. Soc. 115 (2) (1992) 573--580. 
			
			\bibitem{HF} H. Furstenberg, \textit{Disjointness in ergodic theory, minimal sets and a problem in Diophantine approximation}, Math. Sys. Theory 1 (1) (1967) 1--49.
			
			\bibitem{HK} H. Kato, \textit{Continuum--wise expansive homeomorphisms}, Canad. J. Math. 45 (3) (1993) 576--598, https://doi.org/10.4153/CJM-1993-030-4.
			
			\bibitem{Key} H. B. Keynes, J. B. Robertson, \textit{Generators for topological entropy and expansiveness}, Mathematical Systems Theory 3 (1969) 51--59, https://doi.org/10.1007/BF01695625.
			
			\bibitem{TKM} T. K. S. Moothathu, \textit{Stronger forms of sensitivity for dynamical systems}, Nonlinearity 20 (2007) 2115-2126, 10.1088/0951-7715/20/9/006.
			
			\bibitem{Oprocha} P. Oprocha, \textit{Shadowing, thick sets and the Ramsey property}, Ergod. Th. and Dynam. Sys. 36 (5) (2016) 1582--1595, https://doi.org/10.1017/etds.2014.130.
			
			\bibitem{wr} W. L. Reddy, \textit{Point--wise expansion homeomorphisms}, J. Lond. Math. Soc. 2 (2) (1970) 232--236, https://doi.org/10.1112/jlms/s2-2.2.232.
			
			\bibitem{RT} V. Renukadevi, S. Tamilselvi, \textit{Stronger forms of sensitivity in the dynamical system of Abelian semigroup action}, Journal of Dynamical and Control Systems 28 (1) (2022) 151-162, https://doi.org/10.1007/s10883-020-09527-w.
			
			\bibitem{SS} S. Schwartzman, \textit{On Transformation Groups}, Dissertation, Yale University, 1952.
			
			\bibitem{es} E. Shah, \textit{Positively expansive maps on $G-$\emph{spaces}}, J. Indian Math. Soc. 72 (2005)91--97.
			
		\end{thebibliography}
\end{document}